\documentclass[%
 aip,
 amsmath,amssymb,
reprint, onecolumn 
]{revtex4-1}

\usepackage{graphicx}
\usepackage{dcolumn}
\usepackage{bm}

\usepackage[utf8]{inputenc}
\usepackage[T1]{fontenc}
\usepackage{mathptmx}
\usepackage{etoolbox}

\makeatletter
\def\@email#1#2{%
 \endgroup
 \patchcmd{\titleblock@produce}
  {\frontmatter@RRAPformat}
  {\frontmatter@RRAPformat{\produce@RRAP{*#1\href{mailto:#2}{#2}}}\frontmatter@RRAPformat}
  {}{}
}%
\makeatother
\usepackage[letterpaper,top=2cm,bottom=2cm,left=3cm,right=3cm,marginparwidth=1.75cm]{geometry}
\usepackage{amsmath}
\usepackage{graphicx}
\usepackage[colorlinks=true, allcolors=blue]{hyperref}
\usepackage{amscd,amssymb,graphicx,amsmath,bm,mathdots}
\usepackage{multirow,array}
\usepackage{amscd,amssymb,graphicx,amsmath,bm,mathdots}
\usepackage{multirow,array}

\usepackage{arydshln}
\usepackage{arydshln}
\newtheorem{proof}{Proof}
\newtheorem{thm}{Theorem}[section]

\newtheorem{cor}[thm]{Corollary}

\newtheorem{prop}[thm]{Proposition}
\newtheorem{dfn}[thm]{Definition}

\newtheorem{rem}{Remark}

\def \C{\mbox{$\mathbb C$}}
\def \R{\mbox{$\mathbb R$}}
\def \Z{\mbox{$\mathbb Z$}}

\def\and{\mbox{ \rm and }}

\input{xy}
\xyoption{all}

\begin{document}
\preprint{AIP/123-QED}
\title{Deformations of the five dimensional Heisenberg Lie algebra}
\author{Alice Fialowski\footnote{Corresponding author: Alice Fialowski}}
\email{alice.fialowski@gmail.com, fialowski@inf.elte.hu}
\affiliation{ Computer Algebra Department\\
 E\"{o}tv\"{o}s Lor\'{a}nd University HUNGARY,\\ H-1117 Budapest, P$\acute{a}$zm$\acute{a}$ny P$\acute{e}$ter s$\acute{e}$t$\acute{a}$ny 1/C
}%
\author{Ashis Mandal}%
\email{amandal@iitk.ac.in, ashismandal.r@gmail.com}
\affiliation{ Department of Mathematics and Statistics\\
 Indian Institute of Technology Kanpur\\
208016, India
}%


\date{\today}

\begin{abstract}
In this note we explicitly give all the equivalent classes of deformations of the $5$-dimensional Heisenberg Lie algebra $\mathfrak{h}_2$  over $\C$ or $\R$. We show that there are altogether $20$ infinitesimal deformations (families), $18$ of them being extendable to real deformations and $2$ of them are only infinitesimal. 
\footnote{AMS Mathematics Subject Classification (2020) : $17$B$56$, B$81$, B$99$. }
\end{abstract}

\keywords{ Heisenberg algebra, deformation, cohomology of Lie algebra}


\maketitle
\section{Heisenberg Lie algebras}

Heisenberg Lie algebras $\mathfrak{h}_n$ play an important role in both mathematical physics and mathematics. They are defined over $\R$ for every odd integer $2n+1$, with basis given by
\begin{center}
$\{e_1, e_2, \ldots , e_n, e_{n+1}, e_{n+2}, \ldots, e_{2n}, e_{2n+1}\}$,
\end{center}
and nonzero commutators are as follows
\begin{center}
$[e_1, e_{n+1}] = e_{2n+1},$

$[e_2, e_{n+2}] = e_{2n+1}$,

$\vdots$

$[e_n, e_{2n}] = e_{2n+1}$.
\end{center}
We also say that this Heisenberg Lie algebra has rank $n$. 
They are nilpotent Lie algebras, with centre being generated by $e_{2n+1}$ and are also viewed as a central extension of the commutative algebra $\R^{2n}$  by $\R$. In quantum mechanics they are typically used for $n = 1, 2, 3$, but also in the theory of vertex algebras with $n \in \Z$.

It is well-known that they can be represented by upper triangular matrices as follows: 

Consider two elements of the rank $n$ Heisenberg algebra: ${\bf{a}} = \sum_1^{2n+1}a_i e_i$ and  ${\bf{b}} = \sum_1^{2n+1}b_ie_i$. They can be represented by $(n+2)\times (n+2)$ triangular matrices.  Here, the basis elements $e_1,\cdots , e_{2n+1}$ are represented by

  {\small
\[
  e_{1} = \left(
    \begin{array}{@{} c c c c c c @{}}
      0&\multicolumn{0}{@{|}c}{1} & 0 &  \cdots &  0 & 0 \\
      \cdashline{2-5}
      0 & 0 &  0 & 0& 0& \multicolumn{1}{ @{|} c}{\ 0} \\
     0 & \vdots &  \vdots &   \vdots &0 & \multicolumn{1}{  @{|}c}{\ 0} \\
     0 & \vdots &  \vdots & \vdots & 0&  \multicolumn{1}{ @{|} c}{\ 0} \\
     0 & \vdots &  \vdots & \vdots & 0&  \multicolumn{1}{  @{|}c}{\ 0} \\
\cdashline{6-6}
     0 & 0 & 0 & 0 & 0 & 0
    \end{array}
  \right),\quad
  e_{2} = \left(
    \begin{array}{@{} c c c c c c @{}}
      0&  \multicolumn{0}{ @{|} c}{0} & 1 &  \cdots &  0 & 0 \\
      \cdashline{2-5}
      0 & 0 &  0 & 0& 0& \multicolumn{1}{ @{|} c}{\ 0} \\
     0 & \vdots &  \vdots &   \vdots &0 & \multicolumn{1}{ @{|} c}{\ 0} \\
     0 & \vdots &  \vdots & \vdots & 0&  \multicolumn{1}{  @{|}c}{\ 0} \\
     0 & \vdots &  \vdots & \vdots & 0&  \multicolumn{1}{ @{|}c}{\ 0} \\
\cdashline{6-6}
     0 & 0 & 0 & 0 & 0 & 0
    \end{array}
    \right), \ldots
    \medskip                                                                           
\]

\[
    e_{n} = \left(
    \begin{array}{@{} c c c c c c @{}}
      0&  \multicolumn{0}{ @{|}c}{0} & 0 &  \cdots & 1 & 0 \\
      \cdashline{2-5}
      0 & 0 &  0 & 0& 0& \multicolumn{1}{@{|}c}{\ 0} \\
     0 & \vdots &  \vdots &   \vdots &0 & \multicolumn{1}{@{|} c}{\ 0} \\
     0 & \vdots &  \vdots & \vdots & 0&  \multicolumn{1}{@{|} c}{\ 0} \\
     0 & \vdots &  \vdots & \vdots & 0&  \multicolumn{1}{@{|} c}{\ 0} \\
\cdashline{6-6}
     0 & 0 & 0 & 0 & 0 & 0
    \end{array}
  \right), \quad
  e_{n+1} = \left(
    \begin{array}{@{} c c c c c c @{}}
      0&  \multicolumn{0}{@{|} c}{0} & 0 &  \cdots &  0 & 0 \\
      \cdashline{2-5}
      0 & 0 &  0 & 0& 0& \multicolumn{1}{@{|} c}{\ 1} \\
     0 & \vdots &  \vdots &   \vdots &0 & \multicolumn{1}{@{|} c}{\ 0} \\
     0 & \vdots &  \vdots & \vdots & 0&  \multicolumn{1}{@{|} c}{\ 0} \\
     0 & \vdots &  \vdots & \vdots & 0&  \multicolumn{1}{@{|} c}{\ 0} \\
\cdashline{6-6}
     0 & 0 & 0 & 0 & 0 & 0
    \end{array}
  \right), \ldots
  \medskip
  \]
 \[
 e_{2n} = \left(
    \begin{array}{@{} c c c c c c @{}}
      0&  \multicolumn{0}{@{|} c}{0} & 0 &  \cdots &  0 & 0 \\
      \cdashline{2-5}
      0 & 0 &  0 & 0& 0& \multicolumn{1}{@{|} c}{\ 0} \\
     0 & \vdots &  \vdots &   \vdots &0 & \multicolumn{1}{@{|} c}{\ 0} \\
     0 & \vdots &  \vdots & \vdots & 0&  \multicolumn{1}{@{|} c}{\ 0} \\
     0 & \vdots &  \vdots & \vdots & 0&  \multicolumn{1}{@{|} c}{\ 1} \\
\cdashline{6-6}
     0 & 0 & 0 & 0 & 0 & 0
    \end{array}
    \right),\quad
    e_{2n+1} = \left(
    \begin{array}{@{} c c c c c c @{}}
      0&  \multicolumn{0}{@{|} c}{0} & 0 &  \cdots & 0 & 1 \\
      \cdashline{2-5}
      0 & 0 &  0 & 0& 0& \multicolumn{1}{@{|} c}{\ 0} \\
     0 & \vdots &  \vdots &   \vdots &0 & \multicolumn{1}{@{|} c}{\ 0} \\
     0 & \vdots &  \vdots & \vdots & 0&  \multicolumn{1}{@{|} c}{\ 0} \\
     0 & \vdots &  \vdots & \vdots & 0&  \multicolumn{1}{@{|} c}{\ 0} \\
\cdashline{6-6}
     0 & 0 & 0 & 0 & 0 & 0
    \end{array}
  \right) \quad
  \]
  }  
  \medskip
  We get
 \begin{center}
$[{\bf{a}}, {\bf{b}}] = \sum_1^n (a_ib_{i+n} - a_{i+n}b_i)e_{2n+1}$.
\end{center}
\medskip

It seems that the most studied case is the $3$-dimensional Heisenberg Lie algebra $\mathfrak{h}_1$ with basis elements $\{e_1, e_2, e_3\}$.  We also know the dimensions of its adjoint cohomology spaces following \cite{Ma}. Let $h^p$ be the dimension of $ H^p(\mathfrak{h}_n, \mathfrak{h}_n) $. Then

\begin{equation*}
 h^p =\begin{cases} 
1 & \text{if $p= 0$,}\\
~~\\
(2n+1) {2n+1 \choose p} - {2n+1 \choose p+1} - 2n {2n+1 \choose p-1} & \text{if $1\leq p\leq n$,}\\
~~\\
(2n+1)  [ {2n \choose n} - {2n \choose {n-2}} ] -  {2n \choose {n-1}} + {2n \choose {n-3}} & \text{if $p=n+1$,}\\
~~\\
2n  [ {2n \choose {p-1}} - {2n \choose {p+1}} ] -  {2n \choose {p}} + {2n \choose {p+2}} & \text{if $n+2 \leq p \leq 2n+1$}.
\end{cases}
\end{equation*}

In low dimensional cohomology spaces we also have explicit cocycles, see \cite{FP1}. For deformations we need to consider the second cohomology space $H^2(\mathfrak{h}_1, \mathfrak{h}_1)$ which is of dimension $5$ (see third formula for $n = 1$). We have $5$ explicit nonequivalent $2$-cocycles, as follows:
\begin{enumerate}
\item $\phi_1(e_2, e_3) = e_3;$
\item $\phi_2(e_1, e_2) = e_2, \phi_2(e_1,e_3) = -e_3;$
\item $\phi_3(e_1,e_2) = e_3;$
\item $\phi_4(e_1, e_3) = e_1;$
\item $\phi_5(e_1, e_3) = e_2.$
\end{enumerate}
These cocycles define $5$ nonequivalent infinitesimal deformations, $[-,-]_{\mathfrak{h}_1} + t \phi_i$ for $ i = 1, \cdots ,5$. They have no $t^2$ terms in the Jacobi identities written for $[-,-]_{\mathfrak{h}_1} + t \phi_i$. All are extendible, so they are real deformations, see \cite {FM}.
\medskip

Nothing is known about deformations of higher rank Heisenberg Lie algebras. The problem is to compute the second cohomology space with adjoint coefficients and to provide all the representative  cocycles. In this note, we completely describe deformations of the $5$-dimensional Heisenberg Lie algebra $\mathfrak{h}_2$ over $\R$ or $\C$.

\section{Preliminaries}

In this section we present a brief summary of cohomology and deformations of a Lie algebra. For details see \cite{NR, F}.

A Lie algebra is a vector space $\mathfrak{g}$  over a field $\mathbb{K}$, equipped with a $\mathbb{K}$-bilinear operation called Lie bracket $[-,-]: \mathfrak{g} \times  \mathfrak{g}\longrightarrow  \mathfrak{g}$ satisfying
\begin{itemize}
           \item Skew-symmetry: $[a,b] = -[b,a]$ for all $a,b \in  \mathfrak{g}$,
           \item Jacobi-identity: $[a,[b,c]]+[b,[c,a]]+[c,[a,b]] = 0$ for all $a,b,c \in  \mathfrak{g}$.
\end{itemize}

A representation of a Lie algebra $\mathfrak{g}$ or a module over $\mathfrak{g}$ is a vector space $V$ with a Lie algebra homomorphism, say $\rho: \mathfrak{g}\longrightarrow \mathcal{E}nd(V)$, where $\mathcal{E}nd(V)$ is the Lie algebra of endomorphisms of $V$ with the commutator Lie bracket operation.
     
Any Lie algebra $\mathfrak{g}$ is a module over itself, it is known as the adjoint representation of $\mathfrak{g}$. Here we consider the module $ V$ as the underlying vector space and define $\rho : \mathfrak{g} \longrightarrow  \mathcal{E}nd(\mathfrak{g})$ as $\rho (x)(y) = [x,y]$ for $x, y \in \mathfrak{g}$.
Let us recall Lie algebra cohomology with coefficients in a module.
\begin{dfn}
Suppose $\mathfrak{g}$ is a Lie algebra and $A$ is a module over $\mathfrak{g}$. Then a $q$-dimensional cochain of the Lie algebra $\mathfrak{g}$ with coefficients in $A$ is a skew -symmetric $q$-linear map on $\mathfrak{g}$ with values in $A$; the space of all such cochains is denoted by $C^q(\mathfrak{g};A)$. Thus, $C^q(\mathfrak{g};A)= Hom_\mathbb{K}(\Lambda^q \mathfrak{g},A)$.  The differential $$d=d_q:C^q(\mathfrak{g};A)\longrightarrow C^{q+1}(\mathfrak{g};A)$$ is the $\mathbb{K}$-linear map given by the formula
\begin{equation*}
 \begin{split}
dc(g_1,\cdots,g_{q+1})=&\sum_{1\leq s<t\leq q+1}(-1)^{s+t-1}c([g_s,g_t],g_1,\cdots,\hat{g}_s,\cdots,\hat{g}_t,\cdots,g_{q+1})\\
&+\sum_{1\leq s\leq q+1 }(-1)^{s}[g_s,c(g_1,\cdots,\hat{g}_s,\cdots,g_{q+1})],
 \end{split}
\end{equation*}
where $c \in C^q(\mathfrak{g};A)$ and $g_1,\cdots,g_{q+1} \in \mathfrak{g}$.

Then $(C^{*}(\mathfrak{g};A),~ d)$ is a cochain complex called the Chevally-Eilenberg complex of $\mathfrak{g}$ while the corresponding cohomology is referred to as the cohomology of the Lie algebra $\mathfrak{g}$ with coefficients in $A$ and is denoted by $H^q(\mathfrak{g};A)$. For $A= \mathfrak{g}$, with the adjoint action as mentioned before, the cohomology is denoted by $H^*( \mathfrak{g}; \mathfrak{g}).$

\end{dfn}
\begin{dfn}
A formal $1$-parameter deformation or simply a \emph{deformation} of ($\mathfrak{g}$,[-,-]) is a Lie algebra structure $[-,-]_t$ on the $\mathbb{K}[[t]]$-module $\mathfrak{g}[[t]]$, such that    
      $$[x,y]_t = [x,y] + \sum_{n\geq1} t^n \mu_n(x,y)$$
       $\mbox{for all}~ x,y \in \mathfrak{g}$.  Here $\mu_n: \mathfrak{g}\otimes\mathfrak{g} \longrightarrow \mathfrak{g}$  are skew-symmetric linear maps for $n\geq 1$, in turn $\mu_n$ is a $2$-cochain in $C^2(\mathfrak{g},\mathfrak{g})$.
\end{dfn}       
       
The Lie algebra structure $[-,-]_t$ is a \emph{deformation} of $(\mathfrak{g}, [-,-]) $ provided the following equalities hold.
\begin{equation}\label{condition}
\begin{split}
&[a, b]_t = - [b,a]_t~~\mbox{for all}~~a,b \in  \mathfrak{g};\\
& [a,[b,c]_t]_t+[b,[c,a]_t]_t+[c,[a,b]_t]_t = 0 ~~\mbox{for all}~~a,b, c \in  \mathfrak{g}.
\end{split}
\end{equation}
Now expanding both sides of the above equations in  (1) and collecting coefficients of $t^n$, we see that (1) is equivalent to the system of equations
\begin{equation}\label{condition-n}
\begin{split}
~\sum_{i+j=n} \{ \mu_i(x,\mu_j(y,z)) +\mu_i(y, ~\mu_j(z, x)) +
\mu_i(z,~ \mu_j(x,y)) \}=0 ~~\mbox{for}~ x,y,z \in \mathfrak{g}.
\end{split}
\end{equation}
A deformation is \emph{of order $n$} if it satisfies the Jacobi identities up to the $t^n$-terms. A deformation $[-,-]_t$ is said to be a \emph{real} deformation if it satisfies the Jacobi identities for any $t^n$-term.
\begin{rem}
For $n=0$, the condition is equivalent to  the  Jacobi identity of the Lie bracket in $\mathfrak{g}$.
For $n=1$, it is equivalent to $ d \mu_1=0$, in other words $\mu_1$ is a $2$-\emph{cocycle} in $C^2(\mathfrak{g},\mathfrak{g})$. In general, for $n\geq 2$, $\mu_n$ is just a $2$-cochain in $C^2(\mathfrak{g},\mathfrak{g})$. 
\end{rem}

\begin{dfn}
The $2$-cocycle $\mu_1$ is called infinitesimal part of the deformation $[-,-]_t=[-,-]+ t\mu_1$, and we say that $\mathfrak{g}_t = \mathfrak{g} + t\mu_1$ is an \emph{infinitesimal deformation} of the Lie algebra $\mathfrak{g}$. In other words, an infinitesimal deformation is a deformation of order $1$, parameterized by $K[t]/(t^2)$. 
\end{dfn}
  \begin{dfn}
  Suppose $[-,-]_t$ and $ [-,-]^{\prime} _t $ are two deformations of a given Lie algebra $\mathfrak{g}$.
      A formal isomorphism is given by  $\Phi_t=\sum_{i\geq0}\phi_i t^i$, satisfying $\Phi_t [x,y]^{\prime}_t$ = $[\Phi_t(x),\Phi_t(y)]_t$ for all $x,y \in \mathfrak{g}$;
 where $\phi_i: \mathfrak{g}\longrightarrow \mathfrak{g}$ are linear maps with $\phi_0=Id_{\mathfrak{g}}$.  In this case we say that the two deformations are equivalent.
\end{dfn}    
   \begin{dfn} 
  Any deformation of a Lie algebra, which is equivalent to the deformation $\mu_0 = [-,-]$ is said to be a \emph{trivial deformation}.   
\end{dfn}    
  \begin{prop}
An infinitesimal deformation of ($\mathfrak{g}$,[-,-])  is given by $\mu_t= [-,-] + t \mu_1 $ such that $\mu_t$ is a deformation of $\mathfrak{g}$ modulo $t^2$ that is, $\mu_t$ satisfies (\ref{condition-n}) for $n =0, 1$.
\end{prop}     
 The next result shows that infinitesimal deformations are characterised by the second cohomology space of a Lie algebra.   
 \begin{prop}
 Equivalent deformations' infinitesimal parts $\mu_1$ and $\mu_2$ differ only in a coboundary, so the cohomology class of the infinitesimal part of a deformation is uniquely determined by a cohomology class of $H^2(\mathfrak{g},\mathfrak{g})$, and different cohomology classes determine  nonequivalent deformations. 
 \end{prop}  
   
 \begin{proof}
 Let $[-,-]_t$ and $ [-,-]^{\prime} _t $ be two deformations of a given Lie algebra $\mathfrak{g}$  where $[x,y]_t = [x,y] + \sum_{n\geq1} t^n \mu_n(x,y)$ and 
  $[x,y]^\prime_t = [x,y] + \sum_{n\geq1} t^n \mu^\prime_n(x,y)$ $\mbox{for all}~ x,y \in \mathfrak{g}$.
   Suppose $\Phi_t: (\mathfrak{g}[[t]], [-,-]_t) \longrightarrow ( \mathfrak{g}[[t]], [-,-]^\prime_t)$ is an equivalence. Then 
   $$\Phi_t [x,y]_t = [\Phi_t(x),\Phi_t(y)]^{\prime}_t $$ for all $x,y\in \mathfrak{g}$.  Expanding this identity and comparing coefficients of $t$, we get $\mu_1-\mu^\prime_1=\delta \phi_1$. 
\end{proof}
\begin{dfn}
 A Lie algebra is said to be \emph{rigid} if every deformation is trivial.
 \end{dfn}
 
\begin{cor}
If $H^2(\mathfrak{g},\mathfrak{g})=0$ then $\mathfrak{g}$ is a rigid Lie algebra.
\end{cor}

\section{Possible nonequivalent deformations of $\mathfrak{h}_2$}

Infinitesimal deformations of a Lie algebra $\mathfrak{g}$ are in one-to-one correspondence with the elements of the second cohomology space $H^2(\mathfrak{g}, \mathfrak{g})$. In \cite{FP2}, the moduli space of $5$-dimensional complex Lie algebras was described (see also \cite{SW}). There are single Lie algebras, and $2$, $3$ and $4$ parameter families. In order to check the possible deformations of the Heisenberg Lie algebra $\mathfrak{h}_2$, we found several possibilities for deformations of $\mathfrak{h}_2$ with basis $\{e_1, \cdots, e_5\}$. The result in \cite{Ma, FP2} is that dim $H^2(\mathfrak{h}_2, \mathfrak{h}_2)= 20$. 
\smallskip

Recall that the nontrivial brackets in $\mathfrak{h}_2$ are given by
\begin{center}
$[e_1, e_3] = e_5, ~~~ [e_2, e_4] = e_5$.
\end{center}
\smallskip

The classification of $5$-dimensional real Lie algebras are known \cite{M1, M2} (1963),  also see \cite{P} (1976). Complex $5$-dimensional Lie algebras are less, as some of them have  two non-isomorphic real forms. Their moduli space can be found in \cite{FP2}. Our method is as follows. 

We start with the complex Heisenberg Lie algebra $\mathfrak{h}_2$ with basis $\{e_1, \cdots, e_5\}$. Note that the specific property of $\mathfrak{h}_2$ is that the nonzero brackets are for different elements, and the resulting element is the fifth one for both brackets. So we looked for such cases. With a change of basis by permutation, from the list of nonequivalent $5$-dimensional Lie algebras in Table 3 of \cite{FP2} we have the following $8$ candidates denoted  by $d_i$ for $i=1, \cdots, 8$, for possible deformations of $\mathfrak{h}_2$ (those which contain some permutation of the elements  the Heisenberg nonzero brackets). We know that infinitesimal deformations are in one-to-one correspondence with the cohomology classes of $H^2(\mathfrak{h}_2,\mathfrak{h}_2)$. So we have 
 $d_i= [-,-]_{\mathfrak {h}_2}+ t \phi_i$       for $i=1, \cdots, 8$.
 In each case one has to check that the extra nonzero brackets, which define the $2$-cochain $\phi_i$ of $\mathfrak{h}_2$, is a cocycle or not (we need $d\phi_i =0$ for $i=1, \cdots, 8$).
 
 \smallskip 
 First we consider $2$-cocycles $\phi_i$ for the Lie algebra $\mathfrak {h}_2$ which are obtained without involving parameters :
\subsection {} The $2$-cochain $\phi_1$ for the algebra $d_1$ given by $[-,-]_{\mathfrak {h_2}}+ \phi_1$ ( denoted as $d_7$ in \cite{FP2})
\begin{equation*}
 \begin{split}
&\phi_1 (e_1, e_3 ) =   2e_2 - 2 e_3 ; \\
&\phi_1(e_1, e_4 ) = e_3 ; \\
&\phi_1(e_2, e_4 ) = e_2; \\
&\phi_1 (e_3, e_4 ) =  2 e_3 - e_2.
 \end{split}
\end{equation*}
\subsection {}  The $2$-cochain $\phi_2$ for the algebra {$d_2$} (denoted as $d_{10}$ in \cite{FP2})
\begin{equation*}
 \begin{split}
&\phi_2 (e_1, e_3 ) = e_1 ; \\
&\phi_2 (e_1, e_4 ) = - e_2 ; \\
&\phi_2 (e_2, e_3 ) = 2 e_2 ; \\
&\phi_2 (e_3, e_4 ) = - e_4\\
&\phi_2 (e_3, e_5 ) = -3e_5 .
\end{split}
\end{equation*}
\subsection {}  The $2$-cochain $\phi_3$ for the algebra {$d_3$} ( denoted as $d_{17}$ in \cite{FP2})
\begin{equation*}
 \begin{split}
&\phi_3 (e_1, e_4 ) = e_1 ; \\
&\phi_3 (e_2, e_4 ) = e_2 ; \\
&\phi_3 (e_3, e_4 ) = e_3 ; \\
&\phi_3 (e_4, e_5 ) = -2 e_5.
\end{split}
\end{equation*} 
Note that $\phi_1, \phi_2, \phi_3$ are all $2$-cocycles.
\medskip

Now we consider the two-parameter families of $2$-cocycles for the Lie algebra $\mathfrak {h}_2$.

\subsection {}  The $2$-cochain $\phi_4(p:q)$ for the family of algebras {$d_4(p:q)$} ( denoted as $d_9(p:q)$ in \cite{FP2})
\begin{equation*}
 \begin{split}
& \phi_4(e_1, e_4)= (p+q)e_1 ;\\
&\phi_4(e_2, e_3)= - e_1 ;\\
& \phi_4(e_2, e_4)=  q e_2 ;\\
&\phi_4(e_3, e_4)= e_2+ p e_3 ;\\
&\phi_4(e_4, e_5)=  -(2p+q)e_5.
\end{split}
\end{equation*}
The map $\phi$ is a 2-cocycle for arbitrary $p$ and $q$.

\subsection {}  The $2$-cochain $\phi_5(p:q)$ for the family of algebras {$d_5(p:q)$} (denoted as $d_{14}(p:q)$ in \cite{FP2})
\begin{equation*}
 \begin{split}
& \phi_5(e_1, e_3)= p e_1 ;\\
& \phi_5(e_2, e_3)=  e_1 + q e_2 +e_4;\\
&\phi_5(e_3, e_4)= - p e_4 ;\\
&\phi_5(e_3, e_5)= - (p+q)e_5.
\end{split}
\end{equation*}
The map $\phi_5$ is a $2$-cocycle for arbitrary $p$ and $q$.

\subsection {}  The $2$-cochain $\phi_6(p:q)$ for the family of algebras {$d_6(p:q)$ (denoted as $d_{15}(p:q)$ in \cite{FP2})

\begin{equation*}
 \begin{split}
& \phi_6(e_1, e_3)= p e_1 ;\\
& \phi_6(e_2, e_3)=  e_1 + q e_2;\\
&\phi_6(e_3, e_4)= - q e_4 ;\\
&\phi_6(e_3, e_5)=  -2q~e_5.
\end{split}
\end{equation*}
The map $\phi_6$ is a $2$-cocycle for arbitrary $p$ and $q$. 

\medskip

Next we consider the three-parameter families of $2$-cocycles for the Lie algebra $\mathfrak {h}_2$:
\subsection { }  The $2$-cochain $\phi_7(p:q:r)$ for the family of algebras $d_7(p:q:r)$ ( denoted as $d_{12}(p:q:r)$ in \cite{FP2})
\begin{equation*}
\begin{split}
&\phi_7 (e_1, e_3) = pe_1;\\
&\phi_7 (e_2, e_3 ) = e_1 + q e_2 ;\\
&\phi_7 (e_3, e_4 ) = - e_2 - r e_4 ;\\
&\phi_7 (e_3, e_5 ) = -( q + r  ) e_5 .
\end{split}
\end{equation*}
The map $\phi_7$ is a $2$-cocycle for arbitrary $p$, $q$ and $r$. 
\subsection { }  The $2$-cochain $\phi_8(p:q:r)$ for the family of algebras {$d_8(p:q:r)$} ( denoted as $d_5(p:q:r))$ in \cite{FP2})
\begin{equation*}
 \begin{split}
&\phi_8 (e_1, e_3 ) = p e_1 ;\\
&\phi_8 (e_1, e_4 ) = (r- q) e_5 ;\\
&\phi_8 (e_2, e_3 ) = e_1 + q e_2 ;\\
&\phi_8 (e_2, e_4 ) = r e_1 - r (p - q ) e_2;\\
&\phi_8 (e_4, e_5 ) = p(r - q) e_5.
\end{split}
\end{equation*}
The map $\phi_8$ is a $2$-cocycle if and only if $q=r$ or $p=0$.
\medskip

We note that all these Lie algebras (families) are solvable. 

\section {Non-isomorphic complex/real deformations of $\mathfrak{h_2}$}

In this section we decide which of these $8$ cocycles give nonequivalent infinitesimal deformations of $\mathfrak h_2$, and among those, which are extendible (no higher order $t$-terms show up in the Jacobi identities). For this, one has to check the Jacobi identity for the modified Lie bracket $ \mu_t(d_i)= [-,-]_{\mathfrak h_2} + t  \phi_i$ for $i=1, \cdots, 8$. If it satisfies the Jacobi identity, then this resulting infinitesimal deformation is extendible and gives a real deformation. If there are higher order $t$-terms then the deformation is only infinitesimal, and cannot be extended to the second order deformation.

\medskip

{\bf{a)} Single deformations}
\medskip

Here $\mu_t(d_1), \mu_t(d_2)$ and $\mu_t(d_3)$ are nonequivalent real deformations of $\mathfrak{h}_2$, as  no $t^2$-term or higher order terms appear in the Jacobi identities.

\begin{thm}

The cocycle $\phi_1$ gives the following infinitesimal and real deformation 
\begin{equation*}
\mu_t(d_1)  =\begin{cases} 
&[e_1, e_3]_t = e_5 + t (2 e_2 - 2 e_3)\\
&[e_1, e_4]_t = t e_3\\
&[e_2, e_4]_t = e_5 + t e_2\\
&[e_3, e_4]_t = t (2 e_3 - e_2)
\end{cases}
\end{equation*}
The cocycle $\phi_2$  gives the following infinitesimal, and real deformation 
\begin{equation*}
\mu_t(d_2)  =\begin{cases} 
&[e_1, e_3]_t = e_5 + t e_1\\
&[e_1, e_4]_t =  - t e_2\\
&[e_2, e_3]_t = 2t e_2\\
&[e_2, e_4]_t = e_5\\
&[e_3, e_4]_t = -t e_4\\
&[e_3, e_5]_t = -3t e_5
\end{cases}
\end{equation*}
The cocycle $\phi_3$  gives  the following infinitesimal and real deformation 
\begin{equation*}
\mu_t(d_3)  =\begin{cases} 
&[e_1, e_3]_t = e_5\\
&[e_1, e_4]_t = t e_1\\
&[e_2, e_4]_t = e_5 + t e_2\\
&[e_3, e_4]_t = t e_3\\
&[e_4, e_5]_t = - 2t e_5
\end{cases}
\end{equation*}

\end{thm}

Obviously, these three deformations are nonequivalent.

\medskip

{\bf{b) 2-parameter deformations}}
\medskip
\begin{thm}
The family of cocycles $\phi_4(p,q)$ gives the following infinitesimal and real family of deformations
\begin{equation}\label{d4}
\mu_t(d_4)(p:q) =\begin{cases} 
&[e_1, e_3]_t = e_5\\
&[e_1, e_4]_t = (p + q) t e_1\\
&[e_2, e_3]_t = -t e_1\\
&[e_2, e_4]_t = e_5 + q t e_2\\
&[e_3, e_4]_t = t e_2 + pt e_3\\
&[e_4, e_5]_t = - (2p +q)t e_5
\end{cases}
\end {equation}
Thus the family of cocycles $\phi_4(p,q)$ forms a $3$-dimensional subspace in $H^2(\mathfrak{h}_2, \mathfrak{h}_2)$.
\end{thm}
\begin{proof}
  We have to check which among these 2-parameter deformations are nonequivalent. 
We get the following:
\smallskip

The family (\ref{d4}) has three nonequivalent infinitesimal deformations, namely the algebras $\mu_t(d_4)(p:0), \mu_t(d_4)(0:q)$ and $\mu_t(d_4)(0:0)$ are non-isomorphic Lie algebras with nonequivalent $2$-cocycles. We get that an arbitrary $\mu_t(d_4)(p:q)$  Lie algebra can be obtained by a linear combination of the above three cocycles. In turn the following identity holds on evaluating at $(e_i, e_j)$ for $1\leq i,j\leq 5$:
\begin{center}
$\mu_t(d_4)(p:q) = \mu_t(d_4)(p:0) + \mu_t(d_4)(0:q) - \mu_t(d_4)(0:0)$
\end{center}
\smallskip

Note that the ``generic element'' $\mu_t(d_4)(0:0)$ in this family is nilpotent, the others are only solvable. This often happens for solvable families of Lie algebras (see \cite{FP2}).
\smallskip

Both the family $\mu_t(d_4)(p:0)$ and $\mu_t(d_4)(0:q)$, and also the single algebra $\mu_t(d_4)(0:0)$ are extendable,  so they are real deformations. Altogether we got {\bf{3}} nonequivalent real deformations with this family of $2$-cocycles.
\end{proof}
\begin{thm}
The $\phi_5(p:q)$ family of cocycles gives the following infinitesimal deformations \begin{equation*}
\mu_t(d_5)(p:q) =\begin{cases} 
&[e_1, e_3]_t = e_5 + pt e_1\\
&[e_2, e_3]_t = t e_1 + qt e_2 + t e_4\\
&[e_2, e_4]_t = e_5 \\
&[e_3, e_4]_t = - pt e_4\\
&[e_3, e_5]_t = - (p+q) t e_5
\end{cases}
\end{equation*}
The family of cocycles form a $3$-dimensional subspace in $H^2(\mathfrak{h}_2, \mathfrak{h}_2)$.
\end{thm}
\begin{proof}
 We get that the cocycles $\phi_5(p:0), \phi_5(0:q)$ and $\phi_5(0:0)$ are nonequivalent cocycles, and $\phi_5(p:q)$ for arbitrary $p$ and $q$ can be 
expressed with a linear combination of those cocycles:
\begin{center}
$\phi_5(p:q) = \phi_5(p:0) + \phi_5(0:q) - \phi_5(0:0)$
\end{center}
\smallskip

Here again, the singular element $\mu_t(d_5)(0:0)$ is nilpotent, the other algebras are only solvable. All these $3$ deformations are extendable, so we get {\bf{3}} new nonequivalent real deformations.
\end{proof}

\begin{thm}
The $\phi_6(p:q)$ family of cocycles gives the following infinitesimal deformations
\begin{equation*}
 \mu_t(d_6)(p:q)  =\begin{cases}
&[e_1, e_3]_t = e_5 + p t e_1\\
&[e_2, e_3]_t = t e_1 + q t e_2\\
&[e_2, e_4]_t = e_5 \\
&[e_3, e_4]_t = - q t e_4\\
&[e_3, e_5]_t = -2q t e_5
\end{cases}
\end{equation*}
The family of cocycles form a $3$-dimensional subspace in $H^2(\mathfrak{h}_2, \mathfrak{h}_2)$.
\end{thm}
\begin{proof}
 We get that the cocycles $\phi_6(p:0), \phi_6(0,p)$ and $\phi_6(0:0)$ are nonequivalent cocycles, and $\phi_6(p:q)$ for arbitrary $p$ and $q$ is the linear combination of those:
\begin{center}
$\phi_6(p:q) = \phi_6(p:0) + \phi_6(0:q) - \phi_6(0:0)$
\end{center}
\smallskip

Here also, the singular element $\mu_t(d_6)(0:0)$ is nilpotent, the other deformations are only solvable. All these $3$ deformations are extendable, so we get {\bf{3}} new nonequivalent real deformations.
\end{proof}
{\bf{c) 3-parameter deformations}}
\medskip
\begin{thm}

The $\phi_7(p:q:r)$ family of cocycles gives the following infinitesimal deformations 

 \begin{equation*}
\mu_t(d_7)(p:q:r) =\begin{cases} 
&[e_1, e_3]_t = e_5 + t p e_1 \\
&[e_2, e_3]_t = t   e_1 +  r t e_2  \\
&[e_2, e_4]_t = e_5 ;\\
&[e_3, e_4]_t = -te_2 - tr e_4  \\
&[e_3, e_5]_t = -(r+q)t e_5 
\end{cases}
\end{equation*}
The family of cocyles forms a $4$-dimensional subspace in $H^2(\mathfrak{h}_2, \mathfrak{h}_2)$.
\end{thm}
\begin{proof}
We get that the cocycles $\phi_7(0:0:0), \phi_7(p:0:0), \phi_7(0:q:0)$ and $\phi_7(0:0:r)$ are nonequivalent cocycles and they all give real deformations. The family $\phi_7(p:q:r)$ for arbitrary $p, q$ and $r$ is a linear combination of those:
\begin{center}
$\phi_7(p:q:r) = \phi_7(p:0:0) + \phi_7(0:q:0) +\phi_7(0:0:r) -2 \phi_7(0:0:0)$
\end{center}

This gives us {\bf{4}} more nonequivalent real deformations.
\end{proof}

\begin{thm}

The $\phi_8(p:q:r)$ family of cocycles (with the restriction $p=0$ or $q=r$) gives the following infinitesimal deformations 
\begin{equation*}
\mu_t(d_8)(p:q:r)=\begin{cases}
&[e_1, e_3]_t = e_5 + t p e_1 \\
&[e_2, e_3]_t = t   e_1 +  r t e_2  \\
&[e_2, e_4]_t = e_5 +  r t e_1  - r ( p -r ) t e_2 
\end{cases}
\end{equation*}
The family of cocycles forms a $4$-dimensional subspace in $H^2(\mathfrak{h}_2,\mathfrak{h}_2)$.
\end{thm}
\begin{proof}
  We get that the cocycle $\phi_8(0:0:0)$ is the same as $\phi_6(0:0)$, so defines the same real/complex deformation.
On the other hand, $\phi_8(p:0:0)$ and $\phi_8(0:q:q)$ are nonequivalent cocycles, giving extendible real deformations, while there are nonequivalent cocycles $\phi_8(0:q:0)$ and $ \phi_8(0:0:q)$, which only give non-extendable infinitesimal deformations, because the $t^2$-terms appear in the corresponding cases when we write the Jacobi identities.

This gives us another {\bf{2}} more nonequivalent real deformations, and we get {\bf{2}} \emph{infinitesimal deformations} which are \emph{not extendable} to higher order.
\end{proof}

To summarise, we get 20 nonequivalent infinitesimal deformations of $\mathfrak{h}_2$, out of which there are $2$ being  not real deformations.

\subsection{Real Heisenberg algebra $\mathfrak{h}_2(\R)$}
Now let us see whether there are more deformations of the real Heisenberg algebra $\mathfrak{h}_2$. The classification of  \cite{M1, M2, P}  says that there are $6$ non-isomorphic nilpotent $5$-dimensional real Lie algebras. Among those listed ones exactly $2$ can be obtained as deformation of $\mathfrak{h}_2$. (Note that their $A_{5, 5}$ algebra is our $d_6(0:0)$ and their $A_{5,6}$ algebra (also the generic element of their family $A_{5, 26}$ for $p=0, \epsilon=1$) is our $d_7(0:0:0)$). For solvable deformations, we have $3$ single deformed algebras, from the $2$-parameter families we have $3 \times 3 = 9$ nonequivalent deformations, and from the $3$-parameter families we have $8$ new deformations. That gives us all the $20$ cocycles. For the real case, we checked the classification list, and did not find any new candidate for possible infinitesimal deformation of the real $\mathfrak{h}_2$.

\section {Summary}
From the computation for dimension of second cohomology space, we know that there are $20$ nonequivalent infinitesimal deformations of the nilpotent Heisenberg Lie algebra $\mathfrak{h}_2$. Here we gave explicit forms of those possible nonequivalent cocycles. It turned out that among the obtained infinitesimal deformations $18$ are extendible to real deformations, as no higher order terms in $t$ show up, while $2$ of them are only infinitesimal. The number of nonequivalent deformations is indeed $20$. Moreover the deformed Lie algebras are all solvable, and among the solvable families there are the ``generic'' nilpotent ones, among those 2 are nonequivalent: $\mu_t(d_4)(0:0) \cong \mu_t(d_7)(0:0)$, and (with a change of basis vector $e_1'= e_1+e_4$) $\mu_t(d_5)(0:0) \cong \mu_t(d_6)(0:0) =  \mu_t(d_8)(0:0:0)$. 

\medskip
This last observation raises an old question of Fialowski and Penkava: Does a solvable several parameter family of Lie algebras always has a ``generic'' nilpotent element with all parameters being $0$? Our conjecture is yes.
\medskip

{\bf Acknowledgements:} \\
We sincerely thank the anonymous referee for making several useful and important suggestions, which helped to improve the  presentation and clarity of the article.

\bigbreak

\end{document}